\documentclass[	10pt,			% Schriftgröße
				a4paper		% Papiergröße
				,notitlepage	% keine Titelseite
				,abstract
				,english
			]%{article}
			{scrartcl}
\usepackage[
			a4paper
			,top=2cm
			,bottom=4cm
			]{geometry}					% Seitenränder
\usepackage[utf8]{inputenc}				% deutsche Symbole
\usepackage[
			english					% englische
			]{babel} 					%Zeichenkodierung
\usepackage{amsmath}					% mathe zeichen
\usepackage{amsfonts}					% mathe schriftarten
\usepackage{amssymb}					% mathe symbole
\usepackage{enumerate}					% aufzählungspacket
\usepackage[arrow, matrix, curve]{xy}	% kommutative Diagramme
\usepackage[thmmarks]{ntheorem}			% für Kästchen nach Beweis
\usepackage[final]{graphicx}			% Graphiken einbinden
\usepackage{caption}					%Für Bildunterschriften
\usepackage{subcaption}					%mehrere Teilbilder in einem Bild

\usepackage{setspace}
\usepackage{bbm}

\usepackage{arydshln}			%dashed lines

\newcommand{\autor}{Robert Nicholls}					% Autor
\newcommand{\HH}{\mathbb{H} }				% doppel H
\newcommand{\RR}{\mathbb{R} }				% doppel R
\newcommand{\NN}{\mathbb{N} }				% doppel N
\newcommand{\betrag}[1]{\lvert #1 \rvert}	% Betrag
\newcommand{\diff}{\mathop{}\!\mathrm{d}}	%Ableitung
\newcommand{\ie}{i.~e.\ }					%id est abkürzung
\newcommand{\spacedot}{\ \cdot \ }			%platzhalterpunkt

\newcommand{\crit}{\mathrm{crit}}

\theoremstyle{break}							% neue Zeile
\theoremseparator{:}							% Doppelpunkt nach Theorem
\newtheorem{theorem}{Theorem}			% Theorem
\newtheorem{lemma}[theorem]{Lemma}				% Lemma
\newtheorem{corollary}[theorem]{Corollary}		% Korollar
\theorembodyfont{\upshape}						% aufrechter Text

\theoremstyle{nonumberplain}					% keine Nummerierung
\theoremheaderfont{\itshape}							% kursive Überschrift
\theoremsymbol{\ensuremath{\square}}			% Beweis Kästchen

\newtheorem{proof}{Proof}						% Beweis

\begin{document} 		% Anfang des Dokuments

%--------------------Formatierung--------------------
\parindent1em				% Einzug nach einer Leerzeile einstellen

%--------------------Titel--------------------
\begin{center}
\begin{doublespace}
\textbf{\huge{Syzygies for periodic orbits in the restricted three-body problem}}
\end{doublespace}
\end{center}
\begin{center}
\autor
\end{center}

\begin{abstract}
In this paper we show the existence of syzygies for all periodic orbits inside the bounded Hill's region of the planer circular restricted three-body problem with energy below the second critical value.
The proof will follow some ideas of Birkhoff \cite{birkhoff} to compute the roots of partial derivatives of the effective potential.
Birkhoff's methods are extended to higher energies and a new base case is created and shown to fulfil the requirements.
An other step from Birkhoff is scrutinized to continue the statement to all mass ratios.
The final step is achieved by integrating over periodic orbits.
Applying the same methods to Hill's lunar problem delivers similar results in that setting as well.
\end{abstract}
%--------------------Inhaltsverzeichnis---------------------
%\thispagestyle{empty}		% leerer Seitenstil, also keine Seitennummer
%\tableofcontents 			% Inhaltsverzeichnis
%\listoffigures				% Abbildungsverzeichnis
%\listoftables				% Tabellenverzeichnis
%\clearpage

\section{Introduction} \label{section introduction}
Going back to Poincaré, the planar circular restricted three-body problem (PCR3BP) is one of the oldest and most studied simplifications of an $N$-body problem.
This paper presents a proof that under certain conditions all periodic orbits have syzygies.
In view of Birkhoff's conjecture that the retrograde bounds a disc-like global surface of section, one would like to show that the retrograde is the systole.
The result from this work proves the non-existence of periodic orbits without syzygies and thus reduces the complexity of the search for further periodic orbits by one dimension.

An other possible application, which this paper could be a starting point for, is the development of symbolic dynamics for the PCR3BP.
Symbolic dynamics in three symbols has been successfully applied to the Euler problem of two centres in \cite{dullin_montgomery} and the existence of an infinite sequence of syzygies has been proven for every---exept Lagrange's---solution in the general three-body problem with zero angular momentum in \cite{montgomery_infinitely} and \cite{montgomery}.
As soon as one adds angular momentum there are periodic non-collisional Langrangian solutions without syzygies.
However, if the angular momentum is small enough then sequences of syzygies derived by free homotopy classes on the reduced configurations space are realized by periodic solutions \cite{Moeckel_Montgomery} and symbolic dynamics has at least numerically been constructed for special types of orbits in a certain PCR3BP situation by \cite{Wilczak_Zgliczynski_1} and \cite{Wilczak_Zgliczynski_2}.
For the general case of the PCR3BP the existence of syzygies for periodic orbits is in general not true.
In \cite{birkhoff} Birkhoff proved a very helpful statement about the roots of the partial derivatives of the effective potential, which we will use in this work.
We will also extend his main argument to higher energies and finally prove the existence of syzygies for periodic orbits within the bounded Hill's region below the second critical value.
In the preparation for the proof we will mainly use notations and arguments from \cite{holo}, which were slightly adapted and summarized in \cite{master}, to also cover the proof of the statement below the first critical level.

The structure of this work will be the following:
First we will go through a short introduction to the PCR3BP in section \ref{section 3-body problem} and then in section \ref{section preparations} recall some general definitions and facts about the PCR3BP as preparation for the proof.
Since we will be using some elementary Morse theory to formalize certain steps from \cite{birkhoff}, we will also state the required definitions and lemmata.
From there on, we will go through two base cases in section \ref{section base cases}:
The first basically comes directly from \cite{birkhoff} and only achieves the weaker statement below the first critical value, while the second case is new and its continuation covers all energies below the second critical value (in  particular it also covers the first case).
The continuation of these bases cases in section \ref{section continuation} will cover all mass ratios and all stated energies to prove the main result

\begin{theorem}\label{main theorem}
Every periodic orbit of the PCR3BP inside the bounded Hill's region~$\mathfrak{K}_c^b$ for an energy $c<H(L_2),H(L_3)$ below the second critical value has at least two distinct syzygies during each period.

\end{theorem}

\section{The planar circular restricted 3-body problem} \label{section 3-body problem}
The restricted problem of three bodies has been widely studied, so we will only do a very brief introduction here.
For further and more thorough inspection of the equations refer to works such as \cite{hagihara} or \cite{szebehely}.

The PCR3BP is the dynamics of a particle---in our case the moon---attracted by and moving in the same plane as the two primaries---here called the sun and the earth---which we assume to have circular motion around their common centre of mass as a Keplerian solution.
We normalize the masses to be~$0<\mu<1$ for the earth and~$1-\mu$ for the sun.
By introducing rotating coordinates the positions of the primaries become stationary at $(-\mu,0)=s$ and $(1-\mu,0)=e$ and the Hamiltonian becomes autonomous:
\begin{align}
H(q,p)
&= \frac{1}{2}( (p_1 + q_2)^2 + (p_2 - q_1)^2 ) \underbrace{ - \frac{\mu}{ \betrag{q - e} } - \frac{1 - \mu}{\betrag{q - s}} - \frac{1}{2}q^2 }_{ =: V(q)}  \label{r3bp H2form}
\end{align}
Here the effective potential energy, as the part independent of momentum, is denoted as~$V(q)$.
By Hamilton's equations of motion one gets the following second order ODEs
\begin{align}
\ddot{q_1}
&= \hphantom{-} 2\dot{q_2} - \frac{\partial V}{\partial q_1}
\label{r3bp equations 1}
\\
\ddot{q_2}
&= - 2 \dot{q_1} - \frac{\partial V}{\partial q_2}.
\label{r3bp equations 2}
\end{align}
In \cite{birkhoff} Birkhoff uses the alternative potential function
\begin{align}\label{Omega of r3bp}
\Omega(q) := \frac{1}{2} \left( (1-\mu) \cdot \betrag{q-s}^2 + \mu \cdot \betrag{q-e}^2 \right) + \frac{\mu}{\betrag{q-e}} + \frac{1-\mu}{\betrag{q-s}},
\end{align}
which only depends on the distances of the moon to sun and earth.
This potential function only differs from~$V$ by sign and a constant: $\Omega(q) = -V(q) + \mu(1-\mu)/2$, so adjusting the signs in the ODEs will render the same dynamics.
It has the advantage that it makes it easy to find critical points through a transformation we will use in section \ref{subsection critical points and hills regions}.

\section{Preparations for the proof} \label{section preparations}
In this section we will first recall some notations for periodic solutions in general, compute critical points and Hill's regions of the PCR3BP and then prove some elementary lemmata from Morse theory.

In a symplectic manifold~$(M,\omega)$ with an autonomous Hamiltonian~$H$ a periodic orbit is a solution~$x \in C^\infty(\RR,M)$ to the ODE
$$\frac{\diff}{\diff t} x(t) = X_H(x(t)) \qquad \forall t \in \RR,$$
such that there exists a period~$T>0$ for which~$x(T+t)=x(t)$ for all $t \in \RR$.
Here~$X_H$ is the Hamiltonian vectorfield defined implicitly by~$\diff H = \omega(\spacedot, X_H)$.
For every non-trivial---that is non-constant---periodic orbit~$x$ there exists a minimal period $T_x := \min \lbrace T>0 \mid x(T+t)=x(t) \quad \forall t \in \RR \rbrace$ and every period~$T=nT_x$ is just a natural multiple.
In our case---the symplectic manifold arises as the cotangent bundle of an open subset of~$\RR^2$ endowed with the standard symplectic form---a periodic orbit is a smooth map~$x \in C^\infty(\RR, \RR^2 \setminus \lbrace e,s \rbrace)$ solving \eqref{r3bp equations 1} and \eqref{r3bp equations 2} such that $x(T+t) = x(t)$ and thus also $\dot{x}(T+t)= \dot{x}(t)$ for all~$t \in \RR$.
Since the value of the autonomous Hamiltonian is constant along its flow we can assert an energy~$H(x)$ to every periodic orbit.

\subsection{Critical points and Hill's regions in the PCR3BP} \label{subsection critical points and hills regions}
We will now compute the critical points and values of the PCR3BP as we will need to refer back to some steps during the main proof later on.
This part will mainly follow \cite{abraham_marsden} and \cite{holo}.
Critical points of the effective potential and thus of the Hamiltonian are the five Lagrange points, which I will denote by 
$$
\pi_{|\crit(H)}(L_i) =: \ell_i \quad \text{for} \quad i=1, \dots 5,
$$
where
\begin{align} \label{critical points bijection}
\left.
\begin{aligned}
 & \hphantom{\big(} \pi_{|\crit(H)} \colon \crit(H) \to \crit(V) 
\quad \text{with}
\\
&\left(  \pi_{|\crit(H)} \right)^{-1}
\left( q_1 , q_2 \right)
= \left(
\left( q_1 , q_2 \right),
\left( - q_2 , q_1 \right)
\right)
\in \crit(H).
\end{aligned}
\right.
\end{align}
is a bijection of critical points of~$H$ and $V$, given by the footpoint projection of the cotangent bundle.
In this 1 to 1 correspondence also the critical values and the Morse indices (the number of negative eigenvalues of the Hessian) of critical points coincide, as the twisted momentum only adds two positive eigenvalues to the Hessian at this point.
Using the symmetry of~$\Omega$ with respect to the $q_1$-axis to our advantage, we will first search for critical points on the upper half plane and then on the $q_1$-axis by restricting the effective potential.

On the upper half plane we use a transformation to the half strip (see figure \ref{figure upper half-plane diffeo}) by

\begin{figure}
\centering
\includegraphics[scale=1]{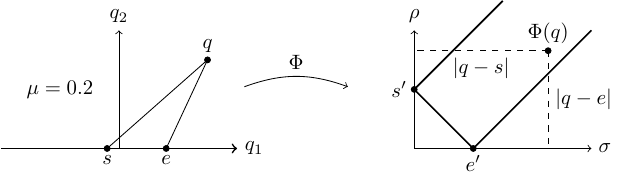}
\caption{Transformation of the upper half-plane (adapted from \cite{abraham_marsden})}
\label{figure upper half-plane diffeo}
\end{figure}

\begin{align*}
\Phi \colon \HH_+ &\to \Theta
\\
q &\mapsto \left( \betrag{q-s}, \betrag{q-e} \right),
\end{align*}
where~$\Theta$ is the diagonal half-strip in the first quadrant of $\RR^2$:
$$
\Theta:=  \left\lbrace (\sigma, \rho) \in (0, \infty)^2 \mid \sigma + \rho > 1, \betrag{ \sigma - \rho } < 1 \right\rbrace
$$
We will now look at our effective potential in the new coordinates~$(\sigma,\rho)$.
Using the alternative form~$\Omega$ of the effective potential from \eqref{Omega of r3bp}, this representation is made very easy as it becomes

\begin{align*}
U&:=\Omega \circ \Phi^{-1} \colon \Theta \to \RR
\\
U(\sigma, \rho)
&=\frac{1}{2} \left( (1-\mu) \sigma^2 + \mu \rho^2 \right) + \frac{\mu}{\rho} + \frac{1-\mu}{\sigma}.
\end{align*}
Critical points of~$U$ can now be found by differentiating:

\begin{align*}
\diff U(\sigma, \rho)
&= \frac{(1-\mu) \left( \sigma^3 - 1 \right)}{\sigma^2} \diff \sigma
+ \frac{\mu\left( \rho^3-1 \right)}{\rho^2} \diff \rho
\overset{!}{=}0,
\end{align*}
thus the only critical point of~$U$ is $(1,1)$.
Computing the Hessian of~$U$ we get
\begin{align*}
%\text{Hess}_U(\sigma, \rho)
%=
%\begin{pmatrix}
%(1-\mu)(1 + \frac{2}{\sigma^3}) & 0 \\
%0 & \mu (1 + \frac{2}{\rho^3})
%\end{pmatrix},
%\quad \text{so} \quad
\text{Hess}_U(1,1)
=
\begin{pmatrix}
3(1-\mu) & 0 \\
0 & 3\mu
\end{pmatrix},
\end{align*}
and our single critical point of~$U$ is a local minimum.

Transforming back to coordinates in the plane, the critical point has the same distance $r_e=r_s=1$ from sun and earth and thus forms an equilateral triangle with the primaries.
We therefore get the critical point
\begin{align*}
\ell_4 := \Phi^{-1}(1,1) = \left( \frac{1}{2}-\mu, \frac{\sqrt{3}}{2} \right),
\end{align*}
which is a local minimum of~$\Omega$.
Recall, that $V(q) = - \Omega(q) + \mu(1-\mu)/2$, so~$\ell_4$ is a local maximum, \ie a point with maximal Morse index 2 of the effective potential for all~$\mu \in (0,1)$.
The Morse index is invariant under the bijection \eqref{critical points bijection}, so $L_4$ becomes a critical point of Morse index 2 of~$H$, \ie only a saddle point.

By reflection on the $q_1$-axis, we get another local maximum of~$V$ at
\begin{align*}
\ell_5:=\left( \frac{1}{2}-\mu, -\frac{\sqrt{3}}{2} \right).
\end{align*}
The Hamiltonian at these critical points takes the same critical value as on~$V$, \ie
\begin{align*}
H( L_4 ) = H(L_5)
= V( \ell_4 ) = V(\ell_5)
%&= -\Omega \left( \frac{1}{2}-\mu, \pm \frac{\sqrt{3}}{2} \right) + \frac{\mu(1-\mu)}{2}
%\\
%&= -U(1,1) + \frac{\mu(1-\mu)}{2}
= \frac{\mu(1-\mu) - 3}{2} .
\end{align*}

The three remaining collinear Lagrange points are attained by restricting the effective potential~$V$ to the $q_1$-axis.
Here we will use the original effective potential~$V$.
Critical points of this one-dimensionally restricted function are then also critical points of the general function as~$V$ is symmetric with respect to the $q_1$-axis.
Therefore, let
\begin{align}\label{u definition}
\begin{split}
u
:=V\big|_{\RR \setminus \{s,e\} }  &\colon
\RR \setminus \{-\mu,1-\mu\} \to \RR
\\
x
&\mapsto
- \frac{\mu}{\betrag{x-(1-\mu)}} - \frac{1-\mu}{\betrag{x+\mu}} - \frac{x^2}{2} .
\end{split}
\end{align}
Finding explicit formulas for the critical points would mean solving quintic equations dependent on~$\mu$.
These can be found in Chapter 10 of \cite{abraham_marsden}.
For our purposes though, it suffices to know in which of the three open intervals they lie and how their energies compare.
By differentiating~$u$ twice, we get
\begin{align*}
u^{\prime \prime} (x) = - \frac{2\mu}{\betrag{x-(1-\mu)}^3} - \frac{2(1-\mu)}{\betrag{x+\mu}^3} - 1 < 0,
\end{align*}
\ie~$u$ is strictly concave.
As~$u$ tends towards~$ - \infty$ for~$x \to -\infty$, $-\mu$, $1-\mu$ as well as $\infty$, we can state that there exist exactly three local maxima of~$u$
\begin{align*}
\ell_3 \in (-\infty, -\mu) \quad \ell_1 \in (-\mu,1-\mu) \quad \ell_2 \in (1-\mu, \infty),
\end{align*}
lying in each of the connected components of the domain as can be seen in figure~\ref{figure restricted u}.

\begin{figure}
\centering
\includegraphics[scale=1]{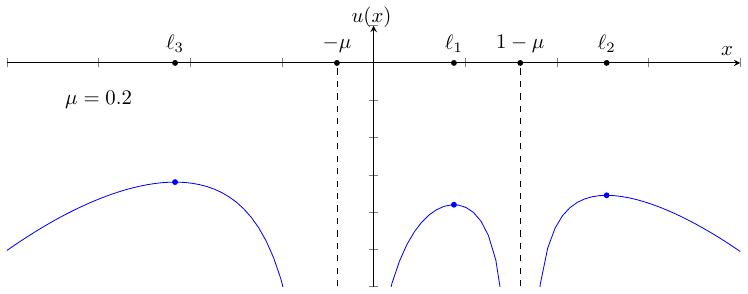}
\caption{The critical points of the effective potential when restricted to the $q_1$-axis}
\label{figure restricted u}
\end{figure}

These critical points of~$V$ are saddle points, \ie have Morse index 1, as proven by topological arguments in \cite[chapter 5, lemma 4.2]{holo} or by computing the Hessian of $\Omega$ in \cite[chapter 10.2]{abraham_marsden}.
Furthermore, this proves, that the pair of critical points~$\ell_4$ and~$\ell_5$ are global maxima of the effective potential, in view of~$V$ tending towards negative infinity at~$e$, $s$ and for large~$\betrag{q}$.

As we have not specified the exact coordinates of the critical points, we can also not state the exact critical energies.
However, we will show in which way they are ordered by following \cite{holo}.
For a point~$-\mu < x < 1-\mu$ denote by
$$\rho := \rho(x) := (1 - \mu) - x = \betrag{x - e} \in (0,1)$$
the distance of this point to earth again.
We will compute the values of the effective potential at the points~$x$ and~$x^\prime := (1 - \mu) + \rho$, which lies symmetrically to~$x$ on the other side of~$e$, by using the restricted function~$u$ from \eqref{u definition}:
\begin{align*}
u(x)-u(x^\prime)
%=& - \frac{\mu}{\rho} - \frac{1-\mu}{1-\rho} - \frac{\left( (1 - \mu) - \rho \right)^2}{2} 
%+ \frac{\mu}{\rho} + \frac{1 - \mu}{1 + \rho} + \frac{\left( (1 - \mu) + \rho \right)^2}{2} 
%%\\
%%=& - \frac{1 - \mu}{1 - \rho} - \frac{1}{2} \left( (1 - \mu)^2 - 2 (1-\mu) \rho + \rho^2 \right)
%%\\
%%&+ \frac{1 - \mu}{1 + \rho} + \frac{1}{2} \left( (1 - \mu)^2 + 2(1-\mu)\rho + \rho^2 \right)
%\\
%%=& (1 - \mu) \left( - \frac{1}{1 - \rho} + \frac{1}{1 + \rho} + 2 \rho \right)
%%\\
%=& \frac{(1-\mu)\left( - (1+\rho) + (1-\rho) + 2\rho \left( 1-\rho^2 \right) \right)}{1-\rho^2}
%\\
=& - \frac{2(1-\mu)\rho^3}{1-\rho^2}
<0
\end{align*}
So the values of~$u$ and thus of~$V$ in between~$s$ and~$e$ are always smaller than symmetrically on the other side of~$e$.
This holds especially for the point~$\ell_1$ and its point opposite of earth~$\ell_1^\prime := (1 - \mu) + \rho(\ell_1)$, so their values are ordered by
$$
V(\ell_1)=u(\ell_1) < u(\ell_1^\prime) \leq u(\ell_2) = V(\ell_2),
$$
as~$\ell_2$ was the maximum of~$u$ for all points larger than~$1-\mu$.
Analogously by estimating the difference of~$u$ for a point~$x \in (-\mu,1-\mu)$ and its opposite to the sun, we get 
$$
V(\ell_1) = u(\ell_1) < u (\ell_3) = V(\ell_3).
$$
Via the identification \eqref{critical points bijection} of critical points of~$V$ and~$H$ retaining the values we have
\begin{align*}
H(L_1) < H(L_2), H(L_3).
\end{align*}
One can furthermore show, that if the sun is strictly heavier than earth, then~$H(L_2)$ is strictly less than~$H(L_3)$ and they are equal exactly at~$\mu=1/2$.
We summarise the results so far in the following lemma:

\begin{lemma}\label{lemma lagrange points}
For all~$\mu\in (0,1)$ there are five critical points of the Hamiltonian~$H$.
They are all saddle points of~$H$.
The pair of symmetric critical points~$\ell_4$ and~$\ell_5$ are global maxima of~$V$, \ie have Morse-index 2 with coordinates
$$
\ell_4 = \left( \frac{1}{2} - \mu, \frac{\sqrt{3}}{2} \right)
\quad
\ell_5 = \left( \frac{1}{2} - \mu, - \frac{\sqrt{3}}{2} \right)
\quad
\text{and}
\quad
H(L_4)  =H(L_5) = \frac{\mu(1-\mu) - 3}{2}.
$$
The remaining three critical points lie along the $q_1$-axis at
$$
\ell_3 < -\mu,
\quad
-\mu < \ell_1 < 1-\mu 
\quad
\text{and}
\quad
\ell_2 > 1-\mu,
$$
and are saddle points of~$V$, \ie their Morse indices are 1 and their energies are ordered by
$$
H(L_1) < H(L_2) , H(L_3) < H(L_4) = H(L_5).
$$
\end{lemma}

With this information, we can now state how the space of all accessible positions changes as the energy decreases from infinity.
This set is known as the Hill's region and is defined as the projection of the energy level set~$\Sigma_c:=H^{-1}(c) \subset T^*(\RR^2 \setminus \lbrace e, s \rbrace)$ from phase space onto configuration space:
$$
\mathfrak{K}_c := \pi ( \Sigma_c ) = \lbrace q \in \RR^2 \setminus \lbrace e, s \rbrace \mid V(q) \leq c \rbrace
$$
From \eqref{r3bp H2form} we see that the Hill's region corresponds to the sublevel set of the effective potential.
In Birkhoff's work \cite{birkhoff} the notion of the oval of zero velocity is used, which denotes the boundary of the Hill's region excluding the primaries or equivalently the level set of the effective potential:
$$\mathfrak{O}_c := \partial \mathfrak{K}_c \setminus \lbrace e,s \rbrace = V^{-1}(c)$$
As the gradient~$\nabla V (q)$ is a normal vector on the level set~$\mathfrak{O}_c$ for every~$q \in V^{-1}(c)$ and regular value~$c$, the tangent to the oval of zero velocity is spanned by
$$
\left\langle
\left( \frac{\partial V}{\partial q_2} (q) , - \frac{\partial V}{\partial q_1}  (q) \right)
\right\rangle_{\RR}
= \nabla V(q)^{\perp}.
$$
\begin{figure}
\centering

\includegraphics[scale=1]{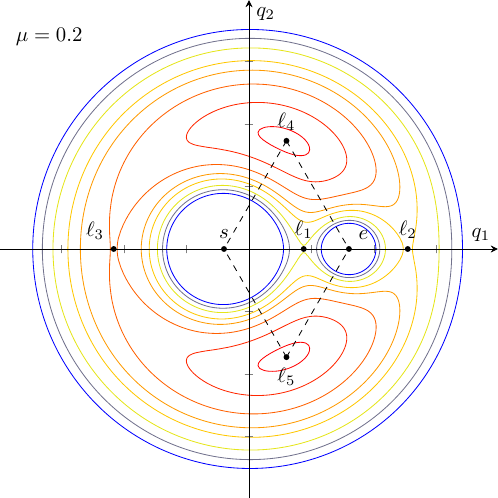}

\caption{Ovals of zero velocity and critical points in the restricted three-body problem}
\label{figure Hills region in r3bp}
\end{figure}
Since~$H$ is symmetric with respect to the $q_1$-axis, so is~$\mathfrak{K}_c$ and~$\mathfrak{O}_c$, and from the type and location of critical points of~$V$ we can state the changes in topology of the Hill's region as the energy varies.
Maxima of~$V$ are attained at~$\ell_4$ and $\ell_5$, so energies~$c \geq V(\ell_4) = V(\ell_5)= (\mu(1-\mu)-3)/2$ above the common critical value result in the Hill's region being all of~$\RR^2\setminus \{e,s\}$, \ie movement is possible  everywhere apart from collisions with the primaries.
At energies below this value, holes will appear around~$\ell_4$ and~$\ell_5$.
These will grow larger with decreasing energy until at the next critical value, they connect to a single hole on the far side of the heavier primary.
If the sun is strictly heavier than earth, \ie $0<\mu<1/2$ then this will be at~$\ell_3$.
This horseshoe shaped hole will grow larger, as~$c$ declines further until its ends meet at the critical point on the far side of the lighter primary, in our case at~$\ell_2$.
At this point, we now have two connected components:
One bounded peanut-shaped component~$\mathfrak{K}_c^b = \pi (\Sigma_c^b)$ containing~$e$ and~$s$ in its closure and one unbounded outer component~$\mathfrak{K}_c^u = \pi (\Sigma_c^u)$.
As the energy declines further yet, the bounded component narrows between the two primaries until they disconnect after surpassing the final critical value of~$\ell_1$.
Finally, for all energies below~$V(\ell_1)$ the Hill's region is divided into three connected components:
A punched disc-like component surrounds each one of the two primaries and will be denoted by~$\mathfrak{K}_c^e = \pi (\Sigma_c^e)$ and~$\mathfrak{K}_c^s = \pi (\Sigma_c^s)$, whereas the unbounded part will remain~$\mathfrak{K}_c^u$.
These changes of the corresponding ovals of zero velocity are visualized in figure \ref{figure Hills region in r3bp} for $\mu=0.2$.
Notations of the oval describing the bounded and unbounded parts will be analogous to the Hill's regions.

\subsection{Some elementary statements about Morse functions} \label{subsection morse theory}
A smooth function is called Morse if all critical points are non-degenerate, \ie if the Hessian at this point is invertible.
As mentioned before the Morse index of a critical point is the number of negative eigenvalues of the Hessian and both non-degeneracy and Morse index of critical points are invariant under coordinate change, thus can be defined for general smooth functions $f \in C^\infty (M, \RR)$ on manifolds.
From these simple definitions we get that all critical points are isolated, which can be easily seen using the Morse Lemma, which states that there exist local coordinates $u$ around a critical point~$p$, such that $u(0)=p$ and
$$
f(u) = f(p) - \sum_{i=1}^{\mu_p} u_i^2 + \sum_{i=\mu_p + 1}^n u_i^2,
$$
where~$\mu_p$ is the Morse index of~$p$.
Obviously the only critical point in this chart is~$p$.
%However the Morse Lemma is not needed in this work, so I will add a short and elementary proof for the discreteness of Morse critical points as well.

\begin{lemma}\label{lemma morse critical points isolated}
Let~$M$ be a smooth manifold of dimension~$n$ and~$f \colon M \to \RR$ a Morse function.
Then all critical points of~$f$ are isolated, \ie for every $p \in \crit f$ there exists an open neighbourhood~$U$ of~$p$, such that $U \cap \crit f = \{ p \}$.
\end{lemma}

%\begin{proof}
%Let $p \in \crit f$ be a critical point of~$f$, $(V, v)$ local coordinates of~$M$ around~$p$ such that $v(0) = p$.
%Define $f_v := f \circ v$ and $v^{-1}(V) =: U_0 \subset \RR^n$, so we have a smooth real valued function $f_v \colon U_0 \to \RR$.
%As $p$ is a critical point of~$f$, \ie $\diff f(p)=0$, we know that the smooth function~$Df_v$ on~$U_0$ attains the value~0 at~0: $Df_v(0)=0$
%Since~$f$ is Morse, the Hessian matrix at~0 is invertible, \ie $H_f^v (0) = H_{f_v}(0) \in GL(n)$.
%Using the inverse function theorem, we know, that $Df_v$ is locally invertible around~$0$:
%\begin{align*}
%&\exists U_1 \subset U_0 \text{ open: } 0 \in U_1 \text{ and } V_1 \subset \RR^n \text{ open: } 0 \in V_1
%\\
%&\exists ! g \colon V_1 \to U_1 \text{ smooth inverse function of } Df_v
%\end{align*}
%This means $Df_v|_{U_1} \colon U_1 \to V_1$ is a diffeomorphism, hence the Jacobian $Df_v(y) \neq 0$ is non-zero for all other points $y \in U_1 \setminus \{0\}$ in this open neighbourhood.
%Set $U:=v(U_1)$, then this open neighbourhood of~$p$ in~$M$ contains no other critical points.
%\end{proof}

As a direct corollary we get that Morse functions on closed Manifolds have finitely many critical points:

\begin{corollary} \label{corollary finite critical points}
Let~$M$ be a closed differentiable manifold and $f$ a Morse function, then the number of critical points of~$f$ is finite.
\end{corollary}

%\begin{proof}
%Lemma \ref{lemma morse critical points isolated} states
%$$
%\forall p \in \crit f  \ \exists U_p \ni p \text{ open: } U_p \cap \crit f = \{ p \}.
%$$
%Define an open cover of~$M$ by $M = \bigcup_{p \in \crit f} U_p \cup (M \setminus \crit f)$.
%Obviously $\crit f$ is closed as the union of isolated points.
%Since~$M$ is compact, there exists a finite subcover.
%But every critical point~$p$ is only in its own neighbourhood~$U_x$ and in no other set of this open cover.
%Hence to remain a cover, every subcover must contain all~$U_x$, and therefore~$\crit f$ is finite.
%\end{proof}

The last lemma we will need as preparation for the theorem, is essentially a step in Birkhoff's proof \cite[chapter 17]{birkhoff}, which he claimed to be obvious.
In the situation of the proof, it is indeed very apparent.
A rigorous proof of a more general statement can be achieved by using the implicit function theorem:

\begin{lemma} \label{lemma number of critical points constant}
Let~$M$ be a closed Manifold and $f \colon \RR \times M \to \RR$ a smooth 1-parameter family of Morse functions.
Then the number of critical points of $f_r := f(r, \spacedot)$ is constant.
\end{lemma}

\begin{proof}
We show, that the function~$ c \colon \RR \to \NN$ whereas $c(r) := \# \crit f_r$, is locally constant:

Let~$r_0 \in \RR$ be a real number, then the number of critical points $\# \crit f_{r_0} =: N < \infty$ is finite by corollary \ref{corollary finite critical points}, because~$f_{r_0}$ is Morse and~$M$ closed.
We will call theses critical points~$\{ p_1, \dots , p_N \} = \crit f_{r_0}$.
The proof, that there exists an open neighbourhood of~$r_0$ in~$\RR$ on which the number of critical points is exactly~$N$, will be split into two steps:
First, we will show, that in an open neighbourhood the number of critical points is at least~$N$.
After that, we precede to prove, that~$c(r)$ is at most~$N$ close to~$r_0$.

For both steps we will use the implicit function theorem, therefore, we will first apply it to our situation:
Working in a chart again, we know, that all critical points~$p_i$ of~$f_{r_0}$ satisfy $Df_{r_0} (p_i) = 0$.
Since~$f_{r_0}$ is Morse, we also know, that the Hessian matrix~$H_{f_{r_0}}(p_i)$ is invertible for all~$i = 1, \dots , N$, hence the requirements of the implicit function theorem are fulfilled, such that we can state:

There exist open neighbourhoods~$U_i$ of~$p_i$, positive~$\epsilon_i > 0$ and unique maps
\begin{gather*}
y_i \colon (r_0-\epsilon_i , r_0 + \epsilon_i) \to U_i,
\qquad \text{such that}
\\
y_i(r_0) = p_i
\quad \text{and}
\\
\forall r \in (r_0 - \epsilon_i , r_0 + \epsilon_i ), p \in U_i:
\quad
y_i(r) = p
\iff
Df_{r_0}(p) = 0.
\end{gather*}
Without loss of generality we can assume $U_i \cap U_j = \emptyset$ if~$i \neq j$.
\\
\underline{Step 1}
 ($ \exists \epsilon >0 : \forall r \in (r_0 - \epsilon , r_0 + \epsilon): \# \crit f_r \geq N$):
\\
Set~$\epsilon := min \{\epsilon_1, \dots , \epsilon_N \}$, then there are~$N$ distinguished critical points at~$y_i(r) \in U_i$, since the open sets~$U_i$ are disjoint.
\pagebreak
\\
\underline{Step 2}
($
\exists \epsilon > 0: \forall r \in (r_0 - \epsilon , r_0 + \epsilon): \# \crit f_r \leq N
$):
\\
Suppose not, then for all~$n \in \NN$ there exists some $r_n \in (r_0 - 1/n , r_0 + 1/n)$ such that there are strictly more than~$N$ critical points of~$f_{r_n}$ on~$M$.
Since the functions~$y_i$ are unique, we can choose a critical point~$p^{(n)}$ for every~$n \in \NN$, such that $p^{(n)} \notin \bigcup_{i=1}^N U_i$.
This gives us a sequence of critical points in~$M^\prime := M \setminus \bigcup_{i=1}^N U_i$.
$M$ and therefore also~$M^\prime$ are compact, so there exists a converging subsequence~$p^{(n_k)} \to p_0 \in M^\prime$, \ie $p_0 \neq p_i$ for $1 \leq i \leq N$.
We can now compute the limit
$$
0=
\lim_{k \to \infty}
\underbrace{
 Df_{r_{n_k}} \left( p^{(n_k)} \right)
 }_{ = 0}
= Df_{r_0} \left( p_0 \right)
$$
and get another critical point~$p_0$ of~$f_{r_0}$, contradicting the assumption, that there were~$N$ critical points to start with.
This concludes the proof, that $\# \crit f_r$ is locally constant, implying it is constant on all of~$\RR$.
\end{proof}
With this Lemma, we are now ready to start to prove the main theorem of this work.

\section{Base cases} \label{section base cases}
The first base case is for an energy below the first critical value and comes directly from Birkhoff \cite{birkhoff}.
It is included here in order to complete the review of Birkhoff's original proof.
The second base case extends this proof to all energies below the second critical value and thus enables a stronger statement with the possibility of more interesting syzygy sequences.
What we want to show for each base case is that there are exactly two vertical tangents to the bounded part of the oval of zero velocity and them being the ones induced by symmetry along the $q_1$-axis.
Recall from the preparation that a vertical tangent corresponds to the derivative~$V_{q_2}:= \partial V /\partial q_2$ of the effective potential with respect to $q_2$ vanishing while $V_{q_1}:= \partial V/\partial q_1$ remains non-zero.

\subsection{Below the first critical value} \label{subsection below first}
Here we consider the limiting cases $\mu = 0$ and $1$, \ie where one of the primaries has zero mass.
In this situation the PCR3BP turns into the rotating Kepler problem, \ie the Kepler problem in rotating coordinates:
\begin{align}
\begin{split}
H(q,p)
&= \frac{1}{2}\left(
(p_1 + q_2)^2 + (p_2 - q_1)^2
\right) - \frac{1}{\betrag{q}} - \frac{1}{2}q^2
\\
&= \frac{1}{2}p^2 - \frac{1}{\betrag{q}} + q_2 p_1 - q_1 p_2
\end{split}
\label{equations rotating kepler}
\end{align}
Note, that this Hamiltonian is no longer a Morse function, since the critical set is the unit circle.
In this simple case, the Hill's regions are either all of the plane minus the origin for energies above the critical value $c=-1.5$, or a punched disc for the bounded Hill's region and the plane minus a larger disc as the unbounded part below the critical value.
Vertical tangents of the ovals of zero velocity obviously only lie along the $q_1$-axis since $V_{q_2}=q_2(1/\betrag{q}^3-1)$ vanishes at $q_2=0$ and the bounded component has $\betrag{q}<1$ making $(1/\betrag{q}^3-1)$ strictly positive.
$V_{q_1}$ can not vanish simultaneously as there are no other critical points.

\subsection{Between the first and the second critical value} \label{subsection below second}
The second base case will use the symmetry for~$\mu=1/2$ and an energy just above the first critical value.
Here, the Hill's region is still reasonably small and we can show, that we can enclose the Hills region in a small enough neighbourhood of the origin, such that within this region the partial derivative~$V_{q_2}$ can not vanish.

At $\mu=1/2$ the effective potential becomes
\begin{align*}
V(q)
= - \frac{1}{2\betrag{q-(\frac{1}{2},0)}} - \frac{1}{2 \betrag{q+(\frac{1}{2},0)}} - \frac{q^2}{2}
\end{align*}
which is additionally symmetric with respect to the $q_2$-axis.
Hence, it suffices to show that on the first quadrant the partial derivative~$V_{q_2}$ does not vanish outside of $\lbrace q_2=0 \rbrace$.
We compute the derivative
\begin{align*}
\diff V(q)
= \left( \frac{ q_1-\frac{1}{2} }{ 2\betrag{q-e}^3 } + \frac{ q_1+\frac{1}{2} }{ 2 \betrag{q-s}^3 } - q_1 \right)\diff q_1
+ q_2 \left( \frac{1}{2 \betrag{q-e}^3} + \frac{1}{2\betrag{q-s}^3} -1 \right) \diff q_2
\end{align*}
to check that $(0,0)$ is the critical point $\ell_1$ and takes the value $V(0)=-2$.
Next, we estimate the value of~$V$ on the unit circle by using the symmetry and restricting to the first quadrant as shown in figure \ref{estimate unit circle}:
\begin{align*}
V(q)
%&= - \frac{1}{2\betrag{q-(\frac{1}{2},0)}} - \frac{1}{2 \betrag{q+(\frac{1}{2},0)}} - \frac{1}{2}
%\\
&\geq - \frac{1}{2\frac{1}{2}} - \frac{1}{2 \frac{\sqrt{5}}{2}} - \frac{1}{2}
%\\
%&= - \frac{1}{\sqrt{5}} - \frac{3}{2}
> - 2
\qquad  \forall \betrag{q}=1 \ q_1,q_2 \geq 0
\end{align*}
\begin{figure}
\centering
\begin{subfigure}{.3\textwidth}
  \centering
  \includegraphics[scale=1]{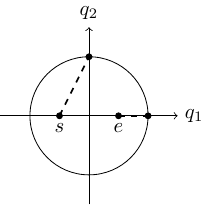}
  \caption{On the unit circle}
  \label{estimate unit circle}
\end{subfigure}
\begin{subfigure}{.3\textwidth}
  \centering
  \includegraphics[scale=1]{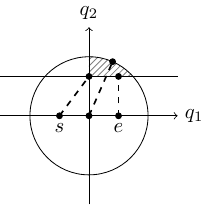}
  \caption{Slicing the lid off}
  \label{estimate lid}
\end{subfigure}
\begin{subfigure}{.3\textwidth}
  \centering
  \includegraphics[scale=1]{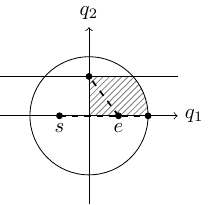}
  \caption{For the set of zeros}
  \label{figure estimate zeros}
\end{subfigure}
\caption{Estimates of $\betrag{q-e}$, $\betrag{q-s}$ and $\betrag{q}$}
\label{figure estimates}
\end{figure}
Since $e$ and $s$ are contained in~$B_1(0)$, there exists some~$\epsilon_0 > 0$ such that the single bounded component of the Hill's region~$\mathfrak{K}_{-2+\epsilon_0}^b \subset \lbrace V(q) \leq -2+\epsilon_0 \rbrace$, which has $e$ and $s$ in its closure, is contained within~$B_1(0)$.
Additionally, we compute that inside the unit ball the $q_2$-values can be restricted to 2/3 by using lower estimates of~$\betrag{q-e}$ and~$\betrag{q-s}$, and an upper estimate of~$\betrag{q}$ as shown in figure \ref{estimate lid}:
\begin{align*}
V(q)
&> - \frac{1}{2 \cdot \frac{2}{3}} - \frac{1}{2 \cdot \frac{5}{6}} - \frac{1}{2}
\\
&= - \frac{37}{20} > -2 \qquad \forall \betrag{q}<1, q_2\geq \frac{2}{3}, q_1 \geq 0
\end{align*}
So, there exists some~$\epsilon_0 \geq \epsilon > 0$ such that~$\mathfrak{K}_{-2+\epsilon}^b$ is contained in $B_1(0) \cap \lbrace \betrag{q_2}< 2/3 \rbrace$.
This concludes the step for restricting the Hill's region and we continue by excluding the set of zeros of the partial derivative~$V_{q_2}$ from this region.

Here, we estimate as in figure \ref{figure estimate zeros}, again using symmetry to restrict to the first quadrant, to get
\begin{align*}
\frac{1}{2 \betrag{q-e}^3} + \frac{1}{2 \betrag{q-s}^3}
&>
\frac{1}{2 \frac{125}{216}} + \frac{1}{2 \frac{27}{8}}
%= \frac{108}{125} + \frac{4}{27}
=\frac{3416}{3375}
>1.
\end{align*}
We can conclude that inside of $B_1(0) \cap \lbrace \betrag{q_2} < 2/3 \rbrace$ the partial derivative $V_{q_2}$
%\begin{align*}
%V_{q_2}(q) = \left( - \frac{1}{2\betrag{q-e}} - \frac{1}{2 \betrag{q-s}} - 1 \right)q_2
%\end{align*}
only vanishes along the $q_1$-axis.
Summarizing the results of this section we state:

\begin{lemma}\label{Lemma base case2}
There exists some~$ \epsilon > 0$ such that inside the bounded Hill's region~$\mathfrak{K}_{-2+\epsilon}^b$ for $\mu=1/2$ the partial derivative~$V_{q_2}$ only vanishes along the $q_1$-axis.
\end{lemma}

\section{Continuation to all mass ratios} \label{section continuation}
The final part of the proof is to continue the base case(s) to all mass ratios and to all energies below the second critical value.
We will discuss Birkhoff's continuation below the first critical value along the way, but focus on the proof of the main (and stronger) statement, which is:

\begin{lemma} \label{Omegayzeros}
Inside the bounded Hill's region~$\mathfrak{K}_c^b$ for energy $c<H(L_2),H(L_3)$ below the second critical value and for all mass ratios $0 < \mu < 1$, the roots of $V_{q_2}(q)$ are precisely the points on the $q_1$-axis.
\end{lemma}

\begin{proof}
Obviously all points of $\lbrace q_2=0 \rbrace$ are zeros of
\begin{align} \label{Omegay}
V_{q_2,\mu}(q)
= q_2 \left( \frac{\mu}{\betrag{q-e}^3} + \frac{1-\mu}{\betrag{q-s}^3} - 1 \right) =: q_2 W_\mu(q)
\end{align}
It remains to show, that inside of~$\mathfrak{K}_c^b$ there are no roots of $W_\mu$.
\\
The first claim is that~$0$ is a regular value:
We prove this by computing where critical points lie and comparing that to the values attained in this region.
The differential of $W_\mu$ is
\begin{align*}
\diff W_\mu (q)
=
-3 \left( \frac{\mu(q_1-1+\mu)}{\betrag{q-e}^5} + \frac{(1-\mu)(q_1+\mu)}{\betrag{q-s}^5} \right) \diff q_1
-3 q_2 \left( \frac{\mu}{\betrag{q-e}^5} + \frac{1-\mu}{\betrag{q-s}^5} \right) \diff q_2
\end{align*}
and thus a critical point requires $q_2$ and simultaneously $\mu(q_1-1+\mu)/\betrag{q-e}^5 + (1-\mu)(q_1+\mu)/\betrag{q-s}^5$ to vanish.
This does not happen if~$\betrag{q_1} \geq 1$, since if $q_1 \geq 1$ then the latter term is bound from below by
\begin{align*}
\frac{\mu^2}{\betrag{q_1-1+\mu}^5} + \frac{1-\mu^2}{\betrag{q_1+\mu}^5}>0
\end{align*}
which is strictly positive.
Analogously if~$q_1 \leq -1$ then the later term is bound from above by
\begin{align*}
\frac{\mu(-2+\mu)}{\betrag{q_1-1+\mu}^5} - \frac{(1-\mu)^2}{\betrag{q_1+\mu}^5} < 0.
\end{align*}
So, critical values can only exist if $\betrag{q_1}<1$ and $q_2=0$ for all mass ratios~$\mu \in [0,1]$.
Estimating the values of~$W_\mu$ in this region gives
\begin{align*}
\frac{\mu}{\betrag{q_1-1+\mu}^3} + \frac{1-\mu}{\betrag{q_1+\mu}^3}-1 > 0.
\end{align*}
This follows by checking the cases $-1<q_1<-\mu$, $-\mu < q_1 <1-\mu$ and $1-\mu < q_1 < 1$.
So the claim is proven that $0$ is a regular value of~$W_\mu$ for all~$\mu$ and thus the preimage~$W_\mu^{-1}(0)$ is a smooth 1-manifold diffeomorphic to $W_0^{-1}(0)=S^1$.
\\
Suppose for some $\mu \in (0,1)$ and $c < V_\mu(\ell_2)$ that $\mathfrak{K}_{\mu,c}^b \cap W_\mu^{-1}(0) \neq \emptyset$, \ie for some point inside the bounded Hill's region the factor~$W_\mu$ of $V_{q_2,\mu}$ vanishes.
Choose a smooth path (figure \ref{figure path})
\begin{align*}
\gamma \colon [0,1] &\to (0,1) \times \RR
\\
t & \mapsto (\mu(t), c(t))
\end{align*}
\begin{figure}
\centering
\includegraphics[scale=1]{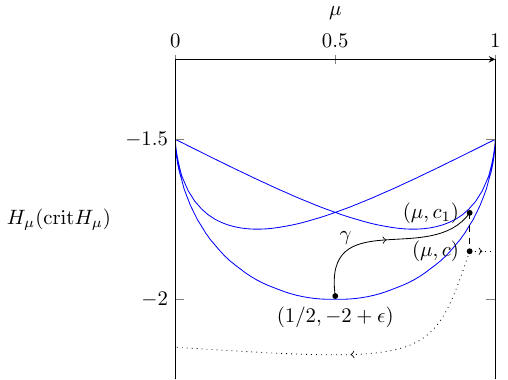}
\caption{Path~$\gamma$ of mass ratios and energies between critical values}
\label{figure path}
\end{figure}
from $\gamma(0) = (1/2, -2+\epsilon)$ to $\gamma(1) = (\mu, c_1)$ such that $c_1 \geq c$ and for all $t\in [0,1]$
$$
H_{\mu(t)}\left( L_1^{\mu(t)} \right) < c(t) < H_{\mu(t)}\left( L_2^{\mu(t)} \right), H_{\mu(t)} \left( L_3^{\mu(t)} \right),
$$
\ie a path between the first and the second critical value from the base case to the mass ratio and (possibly higher) energy of the supposed problematic case.
Note that $\mathfrak{K}_{c, \mu}^b \subset \mathfrak{K}_{c^\prime,\mu}^b$ for $c \leq c^\prime$, so we get $\mathfrak{K}_{\mu,c_1}^b \cap W_\mu^{-1}(0) \neq \emptyset$.
Since $\mathfrak{K}_{\mu(0),c(0)}^b \cap W_{\mu(0)}^{-1}(0)=\emptyset$ by Lemma \ref{Lemma base case2} but $\mathfrak{K}_{\mu(1),c(1)}^b \cap W_{\mu(1)}^{-1}(0) \neq \emptyset$ by assumption and everything deforms smoothly along the path, there must be some $t_0 \in (0,1]$ such that $\partial \mathfrak{K}_{\mu(t_0),c(t_0)}^b \cap W_{\mu(t_0)}^{-1}(0) \neq \emptyset$, \ie the oval of zero velocity intersects the zero level set of~$W$.
Using the symmetry of both $\mathfrak{K}_{\mu,c}^b$ and $W_\mu^{-1}(0)$ we can assume this intersection appears outside of $\lbrace q_2=0 \rbrace$, say $q^0=(q_1^0, q_2^0) \in \partial \mathfrak{K}_{\mu(t_0),c(t_0)}^b \cap W_{\mu(t_0)}^{-1}(0)$ and $q^0_2 \neq 0$.
This implies $V_{q_2, \mu(t_0)}(q^0) = q_2^0 W_{\mu(t_0)}(q^0)=0$, \ie $\mathfrak{O}_{\mu(t_0),c(0)}^b$ has an additional vertical tangent away from the $q_1$-axis at $q^0$.
\\
To continue the first base case below the first critical value one would have to find a smooth path to either $\mu = 0$ (as originally in \cite{birkhoff}; left dottet path in figure \ref{figure path}) or to $\mu=1$.
In order to simply be able to fix the energy (second dotted path) one would additionally have to prove the convexity of the first critical energy level.
From there on one gets an additional vertical tangent to the oval of zero velocity using a similar argument as above.
\\
In order to show that there can not be an additional vertical tangent of the oval of zero velocity to the obvious ones at the intersection with the $q_1$-axis, we write down the smooth dependence of the ovals along the path~$\gamma$ as a smooth 1-parameter family of diffeomorphisms
$$
\phi_t \colon S^1 \to \mathfrak{O}_{c(t),\mu(t)}^b.
$$
So, we can define a smooth 1-parameter family of real valued functions
\begin{align*}
f \colon [0,1] \times S^1 &\to \RR
\\
(t,x) &\mapsto f_t(x):= \pi_1\left( \phi_t(x) \right),
\end{align*}
where $\pi_1 \colon (q_1,q_2) \mapsto q_1$ is the projection along the second coordinate.
\\
We will show next that~$f_t$ is Morse for all~$t \in [0,1]$ in order to apply lemma \ref{lemma number of critical points constant}.
For that, we have to check that every critical point of~$f_t$ is nondegenerate.
Critical points here are points where the oval of zero velocity has vertical tangent, \ie where the partial derivative $V_{q_2}$ vanishes.
A critical point is degenerate if it is also a point of inflexion, \ie if the second derivative $\partial^2 V/ \partial q_2^2 =:V_{q_2^2}$ vanishes as well.
The proof that there are no vertical points will be split into several cases, depending on the location of the point (as also done in \cite{birkhoff})
\\
\underline{Case 1} ($q_2=0$):
\\
If~$q_2=0$ then $V_{q_2^2}$ reduces to
\begin{align}
\frac{\mu}{\betrag{q_1 - (1 - \mu)}^3} + \frac{1-\mu}{\betrag{q_1 + \mu}^3} - 1 \overset{!}{=} 0.\label{ycoeff0}
\end{align} 
\\
\underline{Case 1.1} ($-\mu < q_1 < 1-\mu$):
\\
So, if~$q_1$ lies between~$e$ and~$s$, the denominators in the two fractions of \eqref{ycoeff0} are both strictly smaller than~1, making the left hand side strictly negative and especially non-zero.
\pagebreak
\\
\underline{Case 1.2} ($q_1 > 1 - \mu$):
\\
On the other hand, if~$q_1 > 1 - \mu$ we compute the partial derivative of~$V$ by~$q_1$ to be
\begin{align*}
&V_{q_1}(q_1,0)
= \frac{\mu \left( q_1 - (1-\mu) \right)}{\betrag{q_1 - (1 - \mu)}^3}
+ \frac{(1 - \mu)\left( q_1 + \mu \right)}{\betrag{q_1 + \mu}^3} 
- q_1
\\
&=
(q_1 - (1 - \mu)) \left(\frac{\mu}{\betrag{q_1 - (1 - \mu)}^3} + \frac{1-\mu}{\betrag{q_1+\mu}^3} - 1 \right) + (1-\mu) \left(\frac{1}{\betrag{q_1 + \mu}^3} - 1 \right)
\end{align*}
on the $q_1$-axis.
For vertical points of inflexion, equation \eqref{ycoeff0} gives us
$$
u^\prime (q_1)
=
V_{q_1}(q_1,0)
\overset{!}{=}
(1 - \mu)\left( \frac{1}{\betrag{q_1 + \mu}^3} - 1 \right) < 0.
$$
However, this is exactly the first derivative of the function~$u = V|_{\RR \setminus \{ e, s \}}$ from \eqref{u definition}, which was defined earlier and shown to be strictly monotonically increasing for $q_1 > 1 - \mu$ until it reaches the critical point~$\ell_2$ and only then starts to decline again (see figure \ref{figure restricted u}).
By symmetry of the Hill's regions and the discussion in chapter \ref{subsection critical points and hills regions} about the shape of the Hill's regions there is no part of the bounded Hill's region beyond this critical point since $c < H(L_2)$.
So~$u^\prime (q_1)$ must be strictly positive and so contradicting the assumption that there is a vertical point of inflexion in this case.
\\
\underline{Case 1.3} ($q_1<-\mu$):
\\
This case works analogously to the latter case, as~$V_{q_1}$ becomes
\begin{align*}
V_{q_1}(q_1,0)
&= \frac{\mu \left( q_1 - (1-\mu) \right)}{\betrag{q_1 - (1 - \mu)}^3}
+ \frac{(1 - \mu)\left( q_1 + \mu \right)}{\betrag{q_1 + \mu}^3} 
- q_1
\\
&=
(q_1 + \mu) \left(\frac{\mu}{\betrag{q_1 - (1 - \mu)}^3} + \frac{1-\mu}{\betrag{q_1+\mu}^3} - 1 \right) - \mu \left(\frac{1}{\betrag{q_1 - (1 - \mu)}^3} - 1 \right)
\end{align*}
and for vertical points of inflexion we would have
$$
u^\prime (q_1)
=
V_{q_1}(q_1,0)
\overset{!}{=}
- \mu \left( \frac{1}{\betrag{q_1 - (1 - \mu)}^3} - 1 \right) > 0.
$$
As in case 1.2 this implies that the restricted potential increases, which is only the case up until $\ell_3$.
By $c<H(L_3)$ these points are not reached by any of the bounded Hill's regions~$\mathfrak{K}_c^b$.
\\
From these first cases we can conclude, that there can be no vertical point of inflexion of the oval of zero velocity along the $q_1$-axis inside of~$\mathfrak{K}_c^b$.

Next, we show that a vertical point of inflexion can also not lie at any other point inside the bounded Hill's region.
\\
\underline{Case 2} ($q_2 \neq 0$):
\\
Here, a vertical point of inflexion implies, that by the quotient rule
$$
\frac{\partial}{\partial q_2} \left( \frac{V_{q_2}}{q_2} \right)
= \frac{q_2 V_{q_2^2} - V_{q_2}}{q_2^2}
\overset{!}{=}
0.
$$
But~$V_{q_2}/q_2$ is just the coefficient of~$q_2$ in \eqref{Omegay} and after differentiating there remains
$$
\frac{\partial}{\partial q_2} \left( \frac{V_{q_2}}{q_2} \right)
= \frac{\partial}{\partial q_2} \left(\frac{\mu}{\betrag{q-e}^3} + \frac{1-\mu}{\betrag{q-s}^3} - 1 \right)
= - 3 q_2 \left( \frac{\mu}{\betrag{q-e}^5} + \frac{(1-\mu)}{\betrag{q-s}^5} \right)
\neq 0.
$$
This final contradiction concludes the statement that there can not be a vertical point of inflexion inside the bounded Hill's region~$\mathfrak{K}_c^b$ and thus all critical points of~$f_t$ are non-degenerate, \ie Morse for all~$t$.
Therefore, we can apply lemma \ref{lemma number of critical points constant} and we obtain, that the number of critical values of~$f_t$ is constant for all~$t \in [0,1]$.
Explicitly the number of critical points is exactly two, since $\mathfrak{O}_{-2+\epsilon,1/2}^b$ is a symmetric and connected 1-submanifold of~$\RR$ and critical points only lie on the $q_1$-axis.
In particular there can not be a vertical tangent of~$\mathfrak{O}_{\mu(t_0),c(t_0)}^b$ at $q^0 \notin \lbrace q_2=0 \rbrace$.
\end{proof}

Using lemma \ref{Omegayzeros}, one can now conclude that every periodic orbit inside the bounded Hill's region has at least two syzygies:
Let~$x$ be a non-trivial orbit in~$\Sigma_c^e$ and~$T_x > 0$ its minimal period time---the only trivial periodic orbit is stationary at~$\ell_1$, which lies on the $q_1$-axis, so the statement also holds here.
Then by integrating the second equation~\eqref{r3bp equations 2} of the equations of motion, we get
\begin{align*}
\int_0^{T_x} \ddot{x}_2 (t) \diff t
&=
- 2 \int_0^{T_x} \dot{x}_1 (t) \diff t - \int_0^{T_x} V_{q_2}(x(t)) \diff t
\\
\text{\ie}
0
&=
\int_0^{T_x} V_{q_2}(x(t)) \diff t.
\end{align*}
Since the integral over a smooth function can only be zero if the function itself passes zero we need $V_{q_2}(x(t_0))=0$ for some~$0 < t_0 < T_x$.
By lemma~\ref{Omegayzeros} this is only the case if~$x_2(t_0) = 0$, \ie if~$x$ intersects the $q_1$-axis.
We can integrate the same equation again from~$t_0$ to~$t_0 + T_x$ to get a second intersection during one period.
Note, that these intersections must be transverse, as the sign of~$V_{q_2}$ must change at these syzygies.

\section{Another (easier) application of the same method: Hill's lunar problem} \label{section Hills case}
\begin{figure}
\begin{subfigure}{.49\textwidth}
  \centering
\includegraphics[width=\textwidth]{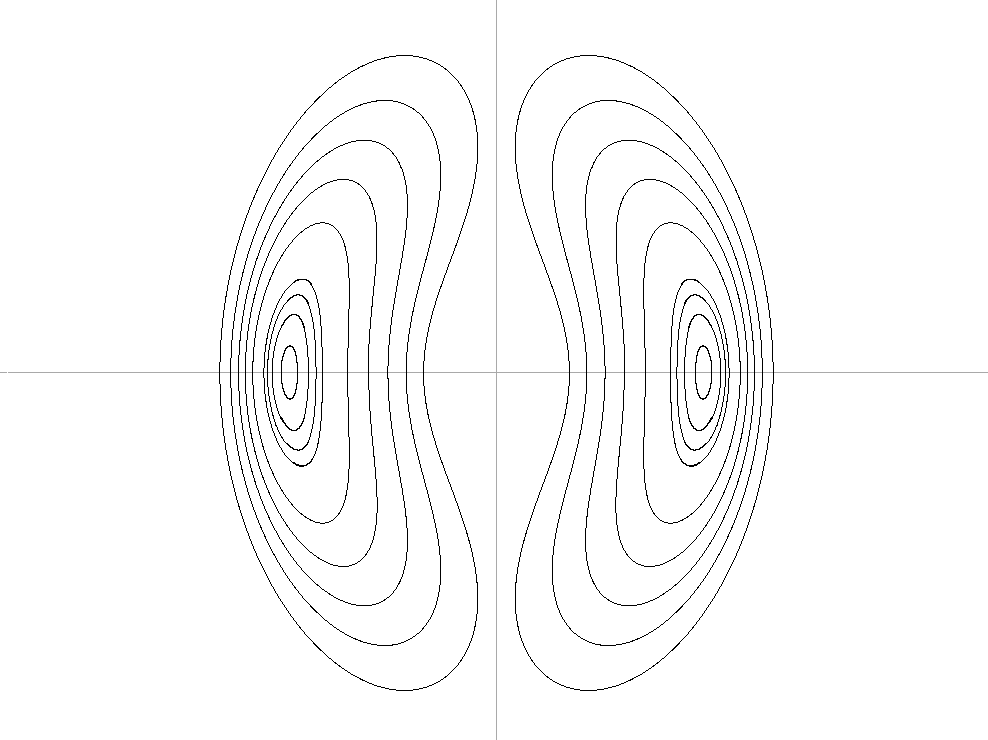}
  \caption{Periodic orbits without quadratures}
	\label{figure a and c}
\end{subfigure}
\begin{subfigure}{.5\textwidth}
  \centering
\includegraphics[scale=1]{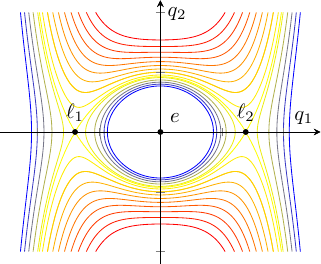}
  \caption{Ovals of zero velocity}
	\label{figure hills problem ovals}
\end{subfigure}
\caption{Hill's lunar problem}
\end{figure}
Hill's lunar problem is a limit case of the PCR3BP by letting the sun get infinitely heavy and at same time infinitely far away (see \cite{hill1}).
The Hamiltonian is given by
\begin{align*}
H(q,p)
&= \frac{1}{2} \betrag{p}^2 + q_2p_1 - q_1 p_2 - \frac{1}{\betrag{q}} - q_1^2 + \frac{1}{2} q_2^2 %\label{Hill Hamiltonian form 1}
\\
&= \frac{1}{2}\left(
(p_1 + q_2)^2 + (p_2 - q_1)^2
\right)
- \frac{1}{\betrag{q}}
- \frac{3}{2}q_1^2
%\label{Hill Hamiltonian form 2}
\end{align*}
and the corresponding second order differential equations are
\begin{align}
\ddot{q}_1
&=\hphantom{-} 2 \dot{q}_2 - \frac{q_1}{\betrag{q}^3} + 3q_1
\label{hills eom1}
\\
\ddot{q}_2
&= - 2 \dot{q}_1 - \frac{q_2}{\betrag{q}^3}.
\label{hills eom2}
\end{align}
We can now do the same integration over \eqref{hills eom2} as in the final step in chapter \ref{section continuation} to get
\begin{align*}
\int_0^T \ddot{x}_2(t) \diff t &= - 2 \int_0^T \dot{x}_1(t) \diff t - \int_0^T \frac{x_2(t)}{\betrag{x(t)}^3} \diff t
\\
0 &= \int_0^T \frac{x_2(t)}{\betrag{x(t)}^3} \diff t
\end{align*}
which only happens if the non-trivial periodic orbit~$x$ passes the $q_1$-axis at some point~$t_0$.
By integrating from~$t_0$ to $t_0+T$ there is a second syzygy in each periodic orbit of Hill's lunar problem and both syzygies are again transverse.
There are two critical points at $(\pm 3^{-1/3},0)$ so the statement also holds for these stationary orbits.

By integrating \eqref{hills eom1} we can also prove the existence of quadratures.
However since the two critical points generate families of periodic orbits away from the $q_2$-axis not all closed orbits have quadratures (figure \ref{figure a and c}).
These are families $a$ and $c$ as classified in \cite{henon}.
One can show---as done in \cite{master} that the bounded Hill's region, \ie the bounded component of possible positions for energies below the common critical value $-3^{4/3}/2$, is contained within the ball of radius~$3^{-1/3}$ around the origin.
So all periodic orbits of Hill's lunar problem with energy below the first critical value have quadratures.
Again one can see the correspondence that additional---in this case horizontal---tangents of the ovals of zero velocity prevent an extension of the statement to higher energies (figure \ref{figure hills problem ovals}).

\section{Remarks} \label{section remarks}
\begin{figure}
  \centering
\includegraphics[scale=1]{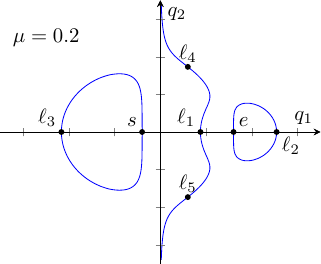}
  \caption{Set of zeros of $V_{q_1}$}
	\label{figure zeros V_q1}
\end{figure}

Since in Hill's lunar problem the horizontal tangents originate in the critical points which lie away from the $q_2$-axis there are direct counterexamples to show the statement can not be extended to a higher energy.
In the PCR3BP however the additional vertical tangents, preventing the proof to go through for higher energies, originate in critical points lying on the $q_1$-axis themselves.
So there are no direct and obvious obstructions for all periodic orbits having syzygies for energies $c<H(L_4)=H(L_5)$.

Similarly for the unbounded Hill's region, the proof as shown in this paper does not work as there can be vertical points of inflexion and so the constructed 1-parameter family of maps on $S^1$ are not Morse.
Indeed, vertical points of inflexion generate additional vertical tangents below the second critical value, which disables our argument of integration.
However, for small enough energies obviously roots of~$V_{q_2}$ only lie along the $q_1$-axis also in the unbounded part and we get syzygies in the same way as before.

However for energies $c>H(L_5)$ there is no hope to prove the existence of syzygies for all periodic orbits since the Lyapunov orbits around the triangular Lagrange points~$L_4$ and~$L_5$ provide direct counterexamples.

Opposed to Hill's lunar problem, in the PCR3BP quadratures with respect to one of the primaries can not be predicted in this manner, since in general neither $\lbrace q_1 = - \mu \rbrace$ nor $\lbrace q_1 = 1-\mu \rbrace$ makes the partial derivative $V_{q_1}$ vanish (although the set of roots comes arbitrarily close for small enough radii; see figure \ref{figure zeros V_q1} and \cite{birkhoff}).

Further interesting questions which can not be answered by this method is whether periodic orbits in the bounded Hill's region below the first critical value have a non-trivial winding number with respect to the respective primary or even if they have syzygies on both sides (\ie if there always exist both solar and lunar eclipse).
Equally unknown is what we can say about sequences of syzygy types for periodic or even arbitrary orbits in the bounded Hill's region below the second critical value.

%--------------------Bibliographie--------------------
%\clearpage										% neue Seite
%\thispagestyle{empty}							% keine Seitenzahl
%\nocite{*}										% alle Literatur aufführen
\addcontentsline{toc}{section}{References}		% ins Inhaltsverzeichnis

\bibliography{Literatur}						% Literaturverzeichnis erstellen

\end{document}